\documentclass[12pt,reqno]{amsart}

\title[Birational Geometry of O'Grady's Six Dimensional Example]{Birational Geometry of O'Grady's Six Dimensional Example 
over the Donaldson-Uhlenbeck Compactification}
\author{Yasunari Nagai}

\address{
Graduate School of Mathematical Sciences, the University of Tokyo, 
3-8-1 Komaba Meguro Tokyo 153-8914, Japan}
\email{nagai@ms.u-tokyo.ac.jp}

\date{July 7, 2010}

\usepackage{amsthm}
\usepackage{amssymb}
\usepackage{amsxtra}
\usepackage[mathscr]{eucal}
\usepackage{enumerate}
\usepackage[all]{xy}
\usepackage[initials,alphabetic,lite]{amsrefs}

\usepackage{times,mathptmx}

\theoremstyle{plain}
\newtheorem{theorem}[subsection]{Theorem}
\newtheorem{lemma}[subsection]{Lemma}
\newtheorem{proposition}[subsection]{Proposition}
\newtheorem{corollary}[subsection]{Corollary}

\newtheorem*{theorem*}{Theorem}
\newtheorem*{corollary*}{Corollary}
\newtheorem*{MainTheorem}{Main Theorem}
\theoremstyle{definition}
\newtheorem{definition}[subsection]{Definition}

\theoremstyle{remark}
\newtheorem*{claim}{Claim}
\newtheorem{remark}[subsubsection]{Remark}

\DeclareSymbolFont{cmletters}{OML}{cmm}{m}{it}
\DeclareSymbolFont{cmsymbols}{OMS}{cmsy}{m}{n}
\DeclareSymbolFont{cmlargesymbols}{OMX}{cmex}{m}{n}
\DeclareMathSymbol{\myjmath}{\mathord}{cmletters}{"7C}
\DeclareMathSymbol{\myamalg}{\mathbin}{cmsymbols}{"71}
\DeclareMathSymbol{\mycoprod}{\mathop}{cmlargesymbols}{"60}
\let\jmath\myjmath
\let\amalg\myamalg

\def\lto{\longrightarrow}

\DeclareMathOperator{\Supp}{Supp}

\DeclareMathOperator{\Pic}{Pic}

\DeclareMathOperator{\Sym}{Sym}
\DeclareMathOperator{\Kum}{Kum}
\DeclareMathOperator{\Spec}{Spec}
\DeclareMathOperator{\Proj}{Proj}
\DeclareMathOperator{\tr}{tr}
\DeclareMathOperator{\pr}{pr}
\DeclareMathOperator{\id}{id}
\DeclareMathOperator{\Aut}{Aut}

\DeclareMathOperator{\Hom}{Hom}

\DeclareMathOperator{\Ext}{Ext}

\DeclareMathOperator{\gr}{gr}
\DeclareMathOperator{\Ker}{Ker}

\DeclareMathOperator{\rank}{rank}
\DeclareMathOperator{\length}{length}

\DeclareMathOperator{\Quot}{Quot}
\DeclareMathOperator{\Def}{Def}

\DeclareMathOperator{\Alb}{Alb}
\DeclareMathOperator{\alb}{alb}

\def\git{/\!\!/}

\renewcommand{\qedsymbol}{Q.E.D.}
\def\noqed{\renewcommand{\qedsymbol}{}}

\begin{document}

\baselineskip 16.6pt
\parskip 6pt

\maketitle
\vspace{-1cm}

\begin{abstract}
We determine the birational geometry of O'Grady's six dimensional example over the Donaldson-Uhlenbeck 
compactification, by looking at the locus of non-locally-free sheaves on the relevant moduli space. 
\end{abstract}

\section*{Introduction}

Let $A$ be an abelian surface whose N\'eron-Severi group is generated by an ample divisor $H$. 
Let $M$ be the moduli space of Gieseker-semistable sheaves on $A$ of rank 2, $c_1=0$, and
$c_2=2$, and $X$ the fiber of the Albanese morphism $M\to \Alb(M)=A\times \hat A$ over the origin. 
O'Grady \cite{OG6} proved that $X$ admits a symplectic resolution 
\[
\pi :\widetilde X\to X
\]
and $\widetilde X$ is an irreducible symplectic K\"ahler manifold of dimension $6$ with 
the second Betti number $8$. This construction gave the fourth new example of 
higher dimensional irreducible symplectic K\"ahler manifold, which we call O'Grady's 
six dimensional example. 

We have another projective birational morphism relevant to $X$, 
namely the morphism to the Donaldson-Uhlenbeck 
compactification
\[
\varphi : M\to M^{DU}
\]
obtained by discarding some algebraic data of non-locally-free sheaves on $M$. 
The exceptional set of $\varphi$ is the locus of non-locally-free sheaves $B_M$. 
Let $\varphi _X$ be the restriction of $\varphi$ to $X$, $B=B_M\cap X$, and $X^{DU}$ the image of $\varphi _X$. 
An analysis of the locus $B$ (or its strict transform $\widetilde B$ on $\widetilde X$) 
was one of the crucial points in \cite{OG6} in proving that 
the second Betti number of $\widetilde X$ is 8, which asserts that $\widetilde X$ is 
not deformation equivalent to the other previously known examples 
of irreducible symplectic K\"ahler manifold. 
The non-locally-free locus $B$  played central role in the works of 
Rapagnetta \cite{R} and Perego \cite{P}, which are about the topology and the singularity 
of O'Grady's six dimensional example, respectively.

On the other hand, not much has been known for the algebro-geometric structure of the example. 
A significant result is due to Lehn--Sorger \cite{L-S}: they showed that O'Grady's resolution $\pi$ is 
nothing but the blowing-up along the singular locus $X_{sing}$. They gave even the local model of 
the singularity of $X$ in terms of nilpotent orbit closure. Thus, we have a complete understanding for 
the resolution $\pi$. It is also noteworthy that Rapagnetta \cite{R} studied a Jabobian-Lagrangian fibration 
on a birational model of $X$.

In this article, we give a complete understanding of the divisor $\widetilde B$, namely, 
we determine explicitly the birational geometry of $\widetilde X$ relative to $X^{DU}$. 

\begin{MainTheorem} 
Under the notation as above,  
\begin{enumerate}[(i)]
\item There exists a projective birational contraction $f : \widetilde X\to X'$ that 
contracts the divisor $\widetilde B$, the strict transform of $B$ on $\widetilde X$, 
and makes the diagram 
\[
\xymatrix @C=10pt  @R=17pt {
& \widetilde X \ar[ld]_{\pi} \ar[rd]^{f}\\
X \ar[rd]_{\varphi _X} & & X'\ar[ld] \\
& X^{DU} }
\]
commutative. 
\item The restriction of $f$ to $\widetilde B$ is a $\mathbb P^1$-bundle with the base $f (\widetilde B)$ 
isomorphic to the product of Kummer surfaces $\Kum (\hat A)\times \Kum (A)$. The singular locus of 
$X'$ coincides with $f (\widetilde B)$ and is a locally trivial family of $A_1$-surface singularities. 
\end{enumerate}
\end{MainTheorem}

This theorem asserts that the birational geometry of $\widetilde X$ is as simple as 
it can be expected. As the value of Beauville-Bogomolov form $q_{\widetilde X}(\widetilde B)=-4$ is negative
(\cite{R}, Theorem 3.5.1), one can easily expect that $\widetilde B$ should be contracted after 
finite sequence of flops. The Main Theorem asserts that actually we need no flop 
to contract the divisor $\widetilde B$. 

The article is organized as the following: we begin with a review of 
the moduli space and the morphism $\varphi$ to the Donaldson-Uhlenbeck compactification.  
Then, we give the statement of the classification of the fibers of $\varphi$ 
(Theorem \ref{strB}) in the first section. 
It is well-known that the Fourier-Mukai functor associated with the Poincar\'e line bundle is 
extremely useful in studying the moduli spaces of sheaves on an abelian surface, 
and it is also the case in our problem. We prove in \S 2 that the Fourier-Mukai functor 
gives a striking explanation to the ``duality phenomenon'' that we will see throughout 
the article (Theorem \ref{FM}). This theorem may be of independent interest. 
The Fourier-Mukai functor will also be used at many technical points in the later sections. 
We establish a GIT theoretic description of the fiber of $\varphi$ in \S 3, which reduces 
the proof of Theorem \ref{strB} to calculation of certain homogeneous invariant rings.  
In \S4, we complete the proof of Theorem \ref{strB} by actually executing the calculation. 
The line of the argument in \S\S 3 and 4 is completely parallel to that of \cite{N} and relying 
on a computer algebra system at some points. 
To obtain from Theorem \ref{strB} the information on $\widetilde B$ that we need to prove our Main Theorem,  
we have to analyze the scheme structure of the intersection $B\cap \Sigma$, which 
will be done in \S 5 using deformation theory. This part is comparatively technical, but plays important role in 
our argument. Here, we again use some computer calculation. 
In \S 6, we prove Main Theorem gathering up the results in the previous 
sections. 

One may ask if one can play the same game also for O'Grady's ten dimensional example \cite{OG10}. 
Theoretically, it is certainly possible; every machinery we use in this article can be 
applied to the case of ten dimensional example (cf. \cite{N}). One main bottleneck 
is that $B\cap \Sigma$ is in fact much more complicated than the current case. 
For that reason, the author has not yet succeeded to complete the program 
for the ten dimensional example up to now. 

\paragraph{\bf Acknowledgment}
The author would like to thank Manfred Lehn for his suggestions. 
Several crucial ideas came out of the discussions with him. 
He would also like to thank Arvid Perego for stimulating discussions. 
The author is supported by Grant-in-Aid for Young Scientists (B) 22740004, 
the Ministry of Education, Culture, Sports, Science and Technology, Japan. 
He is a member of Global COE program of Graduate School of Mathematics, 
the University of Tokyo. 

\section{Non-locally-free locus of O'Grady's six dimensional example}

Let $A$ be an abelian surface with $\dim _{\mathbb R} NS (A)_{\mathbb R}=1$, 
$H$ an ample divisor on it, and $\hat A=\Pic ^0(A)$ the dual abelian surface.  
We consider the moduli space $M$ of Gieseker $H$-semistable sheaves on $A$ with rank $2$, 
$c_1=0$, $c_2=2$. The Albanese morphism of $M$ is given by
\[
\alb _M: M\to A\times\hat A,\quad [E]\mapsto (\sum \mathbf c_2(E), \det E), 
\]
where $\mathbf c_2$ is the chern class map taking value in the Chow ring and 
$\sum$ denotes the summation map $CH^2(A)\to A$. 
The Albanese morphism $\alb _M$ turns out to be a surjective isotrivial family. 
We define
\[
X=\alb _M^{-1}(0,0).
\]
The variety $X$ is of dimension 6, since $\dim M=10$. O'Grady \cite{OG10,OG6}, 
proved that $X$ is singular but admits a symplectic resolution 
$\pi:\widetilde X\to X$, and $\widetilde X$ is irreducible symplectic manifold with 
the second Betti number $b_2(\widetilde X)=8$.  
Later, Lehn--Sorger \cite{L-S} proved that the resolution $\pi$ 
is nothing but the blowing-up along $(X_{sing})_{red}$. 

Let $\Sigma _M$ be the locus of strictly semistable sheaves on $M$. 
By \cite{OG6}, Lemma 2.1.2, every strictly semistable sheaf $[E]\in \Sigma _M$ 
is $S$-equivalent to $\mathfrak m_{p_1}L_1\oplus \mathfrak m_{p_2}L_2$, 
where $p_1, p_2\in A$ and $L_1, L_2\in \Pic ^0(A)$. 
Denote by $\Sigma = \Sigma _M\cap X$ the restriction 
of $\Sigma _M$ to $X$. Then, $[E]\in \Sigma$ if and only if 
$[E]=[\mathfrak m_{p}L \oplus \mathfrak m_{-p}L^{-1}]$. Therefore, we have a stratification  
$\Sigma =\Sigma ^0 \amalg \Sigma ^1$, where
\[
\begin{aligned}
\Sigma ^0&=\{[\mathfrak m_pL\oplus \mathfrak m_{-p}L^{-1}]\in X \mid p\not\in A[2] \mbox{ or } L\not\in 
(\Pic ^0(A))[2]\}, \\
\Sigma ^1&=\{[(\mathfrak m_pL)^{\oplus 2}]\in X \mid p\in A[2] \mbox{ and }L\in (\Pic ^0(A))[2]\}.
\end{aligned}
\]

Let $B_M$ be the locus of non-locally-free sheaves on $M$, namely, we define 
\[
B_M=\{[E]\in M\mid E\mbox{ is not locally free}\}, 
\]
and put $B=B_M\cap X$. Obviously, $\Sigma _M\subset B_M$ and $\Sigma \subset B$. 
$B_M$ can be captured as the exceptional locus of 
the morphism to the Donaldson-Uhlenbeck compactification
\[
\varphi: M\to M^{DU},
\]
(see \cite{H-L}, Chap. 8). We denote by $\varphi _X$ the composition
$X\hookrightarrow M\overset{\varphi}{\to}M^{DU}$. 

\begin{proposition}\label{doubledual}
Let $[E]\in B_M$ and consider its double dual $E^{**}$. Then $E^{**}$ is 
locally free and $\mu$-semistable with $c_1(E^{**})=c_2(E^{**})=0$
and $E^{**}$ is a (possibly trivial) extension of line bundles
\[
0\lto L_1\lto E^{**}\lto L_2\lto 0, 
\]
where $L_1,L_2\in \Pic ^0(A)$. If $[E]\in B$, we have $L_2\cong L_1^{-1}$. 
\end{proposition}

\begin{proof}
This is exactly \cite{OG6}, Lemma 4.3.3, if $E$ stable. 
It is easier to see the case in which $[E]$ is strictly semistable; if $[E]\in B$ is 
strictly semistable, then $E$ is S-equivalent to $\mathfrak m_{p}L_1\oplus \mathfrak m_{q}L_2$, 
so that $[E^{**}]=[L_1 \oplus L_2]$. As $\det E=L_1\otimes L_2$, if $[E]\in B$, namely, if 
$\det E=\mathcal O_A$, we must have $L_2\cong L_1^{-1}$. 
\end{proof}

\subsection{}\label{descrvarphi}
This proposition implies that we have the following 
short exact sequence for each $E\in B_M$; 
\[
0\lto E\lto E^{**}\lto Q(E)\lto 0,
\]
where $Q(E)$ is of length $c_2(E^{**})=2$. We associate to $Q(E)$ 
a 0-cycle $c(Q(E))\in \Sym ^2(A)$ by
\[
c (Q(E))=\sum _{p\in A} \length(Q(E)_p)\cdot p .
\]
The morphism 
$\varphi$ is given by the correspondence (\cite{H-L}, Chap. 8)
\[
E\mapsto \gamma (E):=(\gr (E^{**}), c (Q(E))). 
\]
Therefore, we know that 
\[
\varphi (B_M)\cong \Sym ^2(\hat A)\times \Sym ^2(A).
\]
If $[E]\in B$, then $\gr (E^{**})$ is of the form $L\oplus L^{-1}$ and 
$\gamma (Q(E))=p+(-p)$, so we know that 
\[
\varphi _X (B)\cong (\hat A/\{\pm 1\}) \times (A/\{\pm 1\})
\]
as  $(\hat A/\{\pm 1\}) \times (A/\{\pm 1\})$ can be identified in $\Sym ^2(\hat A)\times \Sym ^2(A)$ 
with the image of the product of anti-diagonals in $\hat A^2$ and $A^2$. 
In the following, we determine every fiber of the restriction
\[
\varphi _{X|B}:B\to \varphi (B).
\]

\begin{theorem}\label{strB}
Let $\gamma =([L], [p])\in (\hat A/\{\pm 1\}) \times (A/\{\pm 1\})$ and 
$B_{\gamma}$ the fiber $\varphi_{X|B}^{-1}(\gamma)$ with the reduced 
structure. 
\begin{enumerate}[(i)]
\item If neither $L$ nor $p$ is 2-torsion, $B_{\gamma}\cong \mathbb P^1$. 
The intersection $B_{\gamma}\cap \Sigma$ consists of two points
$[\mathfrak m_pL\oplus \mathfrak m_{-p}L^{-1}]$ and $[\mathfrak m_{-p}L\oplus \mathfrak m_pL^{-1}]$.
\item If exactly one of $L$ and $p$ is 2-torsion, $B_{\gamma}\cong \mathbb P^2$. 
The intersection $B_{\gamma}\cap \Sigma$ consists of one point, which is
$[\mathfrak m_pL\oplus \mathfrak m_{-p}L]$ if $L^{\otimes 2}\cong \mathcal O_A$, and 
$[\mathfrak m_pL\oplus \mathfrak m_pL^{-1}]$ if $p$ is a 2-torsion point on $A$. 
\item If both of $L$ and $p$ are 2-torsion, $B_{\gamma}$ is a cone over a smooth quadric 
surface in $\mathbb P^4$. The intersection $B_{\gamma}\cap \Sigma$ is the vertex of the cone, 
which corresponds to $[(\mathfrak m_pL)^{\oplus 2}]$. 
\end{enumerate}
\end{theorem}

\begin{remark}
For the time being, we regard $B_{\gamma}\cap \Sigma$ only as a set. 
Actually, the scheme structure of the intersection is non-reduced in the 
cases (ii) and (iii) (Theorem \ref{centerofblowingup}), 
which will be important in the proof of our Main Theorem. 
We will come back to this point in \S 4. 
\end{remark}

The proof of the theorem goes in the same way as in \cite{N}.
Namely, we describe $B_{\gamma}$ as a projective GIT quotient of
certain affine variety (\S 3) and get the set of projective equations for $B_{\gamma}$ 
by actually calculating the associated invariant ring (\S 4). 

\begin{remark}
O'Grady \cite{OG6} already studied the fibration $\varphi _{X|B} : B\to \varphi(B)$ 
in order to determine the fundamental group and the second Betti number 
of the holomorphic symplectic manifold $\widetilde X$. His calculations in \emph{op. cit.}, 
especially in \S 5, hint that the fiber of $\varphi _{|B}$ should be just as in Theorem \ref{strB}, 
and even gives a faint view toward Main Theorem, although he never claimed them explicitly. 
Our approach to the theorem will give an easy and conceptually clarified explanation of 
the phenomenon. 
\end{remark}

\section{Fourier-Mukai transforms}

Before moving on to the proof of Theorem \ref{strB}, 
we prepare an elementary result about 
Fourier-Mukai transforms associated with the Poincar\'e line bundle (Theorem \ref{FM}). 
Our reference for this section is \cite{Y}, \S2. See also \cite{Muk}. 

Let $\mathcal P$ be the Poincar\'e line bundle on $\hat A\times A$. The Fourier-Mukai 
functor $\Phi : D(A)\to D(\hat A)$ defined by
\[
\Phi (a)=\mathbb R\pr _{\hat A\, *} (\mathcal P \otimes \pr _{A}^* (a)). 
\]
gives an equivalence between 
the bounded derived categories of coherent sheaves (\cite{Muk}). 
We define the dualizing functor $\mathbb D_{\hat A}:D(\hat A)\to D(\hat A)_{op}$ by
\[
\mathbb D_{\hat A}(\hat a)
=\mathbb R\mathcal Hom _{\mathcal O_{\hat A}}(\hat a, \mathcal O_{\hat A})[2]
\]
and define $\Phi^D = \mathbb D_{\hat A}\circ \Phi$, following \cite{Y}, \S2.  We likewise define  
$\hat \Phi : D(\hat A)\to D(A)$ by 
\[
\hat \Phi (\hat a)=\mathbb R\pr _{A\, *} (\mathcal P \otimes \pr _{\hat A}^* (\hat a))
\]
and $\hat \Phi ^D=\mathbb D_{A}\circ \hat \Phi$. If $H$ is an ample divisor 
whose class generates $NS(A)$, $\hat H=\det (-\Phi (H))$ gives an ample divisor 
that generates $NS(\hat A)$. There is a spectral sequence 
\begin{equation}\label{FMSS}
E_2^{p,q}=H^p(\hat \Phi ^D(H^{-q}(\hat \Phi ^D(E))))\Rightarrow 
\begin{cases}
E & (p+q=0)\\
0 & (\mbox{otherwise})
\end{cases}
\end{equation}
for a coherent sheaf $E$ on $A$ (see (2.14) of \cite{Y}). 
We say that a coherent sheaf $E$ on $A$ satisfy WIT (abbreviation for ``weak index theorem'') of index $i$ 
with respect to $\Phi ^D$ if the cohomology sheaves $H^j(\Phi ^D(E))$ vanishes for every $j\neq i$. 

One of the most fundamental and elementary observations for $\Phi ^D$ is the following

\begin{proposition}[cf. \cite{Muk}, Example 2.6]\label{SS-LB}
The skyscraper sheaf $\mathcal O_p$ for $p\in A$ 
(resp. a numerically trivial line bundle $L\in \Pic ^0 (A)$ on $A$) 
satisfies WIT for index 2 and $H^2(\Phi ^D(\mathcal O_p))$ is a numerically trivial 
line bundle (resp. a skyscraper sheaf) on $\hat A$.
\end{proposition}

\begin{proof}
The proof is the same as in \emph{op. cit.} One should note that $\Phi^D$ is ``dualized'' so that 
we have to look at $\Ext^i (\mathcal O_p\otimes \mathcal P_y, \mathcal O_A)$, 
where $\mathcal P_y=\mathcal P_{\{y\}\times A}$ for $y\in \hat A$, and so on. 
\end{proof}

\begin{lemma}\label{FMlemma}
\begin{enumerate}[(i)]
\item (cf. \cite{Muk}, Example 2.9) Let $\mathscr E_r$ be the set of isomorphism classes of 
vector bundles $E$ on $A$ of rank $r$ that admits a full flag of sub-bundles 
\[
0=E_0\subset E_1\subset \cdots \subset E_r=E
\]
such that $E_i/E_{i-1}\in \Pic ^0(A)$. Let $\mathscr A_r$ be the set of isomorphism classes of 
artinian $\mathcal O_{\hat A}$-modules of length $r$. Then, every $E\in \mathscr E_r$ 
(resp. $M\in \mathscr A_r$) satisfies WIT of index $2$ with respect to the functor $\Phi ^D$ (resp. 
$\hat \Phi ^D)$. The correspondence $E\mapsto H^2(\Phi ^D(E))$ gives a bijection  
$\mathscr E_r\to \mathscr A_r$, 
whose inverse is given by $H^2(\hat \Phi ^D( - ))$. 
Particularly, in the case $r=2$, $\Phi ^D$ gives a one to one correspondence 
\[
\left\{\begin{matrix}\mbox{extensions $0\to L_1\to F\to L_2\to 0$}\\ 
\mbox{with $L_i\in \Pic ^0(A)$}
\end{matrix}\right\}
\overset{\sim}{\to}
\left\{\begin{matrix}
\mbox{artinian $\mathcal O_{\hat A}$-modules}\\ 
\mbox{of length 2}
\end{matrix}\right\}.
\] 
\item If $N$ is torsion-free sheaf on $A$ of rank 1 and $c_1(N)=0$, $c_2(N)=k>0$, 
WIT of index 1 holds for $N$ with respect to $\Phi ^D$, 
and $H^1(\Phi ^D(N))$ is of rank $k$, $c_1=0$, $c_2=1$.
\end{enumerate}
\end{lemma}

\begin{proof}
(i) By induction on $r$. The case of $r=1$ is nothing but the previous proposition. 
Every $E\in \mathscr E_{r+1}$ fits into 
\[
0\lto E'=E_r\lto E\lto L\lto 0 
\]
with $E'\in \mathscr E_r$ and $L\in \Pic ^0(A)$. Then, we get the exact sequence
\[
H^{i-1}(\Phi ^D(L)\to H^i(\Phi ^D(E'))\to H^i(\Phi ^D(E))\to H^i(\Phi ^D(L) \to H^{i+1}(\Phi ^D(E')), 
\]
so by the induction hypothesis and the previous proposition, $E$ also satisfies 
WIT of index $2$ and $H^2(\Phi ^D(E))$ is an artinian module of length $r+1$. 
The converse correspondence is proved in the same way. 

\noindent (ii) Since we can write $N=I_Z\ L$ with $Z\subset A$ a 0-dimensional subscheme of 
length $k$ and $L$ a numerically trivial line bundle on $A$, we have
\[
0\lto H^1(\Phi^D(N))\lto H^2(\Phi ^D(\mathcal O_Z))\lto H^2(\Phi ^D(\hat L))\lto 
H^2(\Phi ^D(N))
\]
where the last term is 0 because $\Ext ^2(I_Z\ L\otimes \mathcal P_y, \mathcal O_{A})
=H^0(I_Z(L\otimes \mathcal P_y))^{\vee}=0$ for every $y\in \hat A$.  
\end{proof}

The following theorem not only plays an important role in the sequel, but also 
would be of independent interest. 

\begin{theorem}\label{FM}
\begin{enumerate}[(i)]
\item The functor $\Phi ^D$ induces an isomorphism 
\[
\alpha : M\overset{\sim}{\to} \hat M, 
\]
where $\hat M$ is the moduli space of $\hat H$-semistable sheaves of 
rank $2$, $c_1=0$, and $c_2=2$ on $\hat A$. 
\item The isomorphism $\alpha$ fits into the 
commutative diagram 
\[
\xymatrix{
M \ar[r]^{\alpha }\ar[d]_{\alb _M} & \hat M \ar[d]^{\alb _{\hat M}}\\
A\times \hat A \ar[r]^{\beta} & \hat A\times A
} \lower46pt\hbox{,}
\]
where $\beta$ is defined by the correspondence in Proposition \ref{SS-LB}. 
\item $\Phi ^D$ preserves 
the non-locally-free locus and the strictly semistable locus, namely
$\alpha (B_M)=B_{\hat M}$ and $\alpha (\Sigma _M)=\Sigma_{\hat M}$. 
Moreover, $\alpha$ preserves the fiber of 
$\varphi _{|B_M}:B_M\to \Sym ^2(\hat A)\times \Sym ^2(A)$ in the sense that
\begin{equation}\label{FMBgamma}
\xymatrix{
B_M \ar[r]^{\alpha} \ar[d]_{\varphi _{|B_M}} 
& B_{\hat M}\ar[d]^{\varphi _{|B_{\hat M}}} \\
\Sym ^2(\hat A) \times \Sym ^2(A) \ar[r]^{\Sym ^2 \beta} & \Sym ^2(A)\times \Sym ^2(\hat A)
}
\end{equation}
is commutative. 
\end{enumerate}
\end{theorem}

Immediately from this theorem, we obtain the following 

\begin{corollary}
Let $\hat X=\alb _{\hat M}(0,0)$, $\hat B=B_{\hat M}\cap \hat X$, and 
$\hat \Sigma=\Sigma _{\hat M}\cap \hat X$.  Then, 
$\Phi ^D$ induces an isomorphism $\alpha : X\to \hat X$ such that 
$\alpha (B)=\hat B$, $\alpha (\Sigma)=\hat \Sigma$, 
and $\alpha (B_{\gamma})=B_{\hat \gamma}$, 
where $\hat \gamma =(\Sym ^2 \beta)(\gamma)$. If $A$ is principally 
polarized, $\alpha$ is a non-trivial involution on $X$, which induces 
an involution on $\widetilde X$. 
\end{corollary}

Note that the Mukai vector of the sheaves $E$ in $M$ is $(2,0,-2)$. 
It is easy to verify that the Mukai vector of $\hat \Phi ^D(E)$ is 
$(-2,0,2)$ (\cite{Y}, (3,2)).  Therefore, if $E$ satisfies WIT for index 1, 
the Mukai vector of $\hat E=H^1(\hat \Phi ^D(E))$ is $(2,0,-2)$, i.e., 
$\rank \hat E=2$, $c_1(\hat E)=0$, $c_2(\hat E)=2$.   
The essential part of Theorem \ref{FM} is summarized as the following 

\begin{proposition}\label{FM'}
A semistable sheaf $E$ on $A$ of rank 2, $c_1=0$, $c_2=2$ 
satisfies WIT for index 1 with respect to $\Phi ^D$.  
If $E$ is locally free $\mu$-stable (resp. non-locally-free stable, resp. strictly semistable), 
then so is $\hat E=H^1(\Phi ^D(E))$. 
\end{proposition}

\begin{proof}
First, we consider the case in which $E\in M$ is $\mu$-stable vector bundle. 
The proof follows the argument of \cite{Y}, \S 3, but our case is much easier.  
As $E$ is $\mu$-stable, $E$ has no non-trivial morphism $E\to \mathcal P_y^{-1}$. 
Therefore, 
we have $\Hom (E\otimes \mathcal P_y,\mathcal O_A)=0$ for any $y\in \hat A$. 
Similarly, we have $\Ext ^2(E\otimes \mathcal P_y,\mathcal O_A)=0$ by $\mu$-stability 
and Serre duality. As Riemann-Roch infers that 
\[
\chi (E\otimes \mathcal P_y,\mathcal O_A)=\langle (2,0,-2),(1,0,0)\rangle =-2, 
\]
where $\langle \cdot ,\cdot \rangle$ stands for the Mukai pairing (see \cite{Y}, \S1, for 
example), we have $\dim \Ext ^1(E\otimes \mathcal P_y,\mathcal O_A)=2$ constantly 
in $y\in \hat A$. This shows that $H^i(\Phi ^D(E))=0$ for $i\neq 1$ and
$H^1(\Phi ^D(E))$ is a vector bundle of rank 2. 

If $\hat E=H^1(\Phi ^D(E))$ is not 
$\mu$-stable, we have  a sub-line bundle $N\hookrightarrow \hat E$ such that 
$\deg N= N\cdot H\geqslant 0$. Here, $N$ satisfies WIT for index 2. 
Noting that $H^2(\hat \Phi ^D(\hat E))=0$ by the spectral sequence \eqref{FMSS},  
we get $H^2(\hat \Phi ^D(N))=0$ 
from the long exact sequence of cohomology
\[
H^2(\hat\Phi ^D(\hat E))\to H^2(\hat \Phi ^D(N))\to 0, 
\]
which is a contradiction. Therefore, $\hat E$ must be $\mu$-stable. 

Next, we treat the case where $E$ is not locally free. As we saw in \S1, 
$E$ fits into the short exact sequence
\[
0\lto E\lto E^{**}\lto Q(E)\lto 0, 
\]
where $E^{**}$ and $Q(E)$ are realized as extensions
\begin{gather*}
0\lto L_1\lto E^{**}\lto L_2\lto 0 \quad (L_1,L_2\in \Pic^0(A)),\\
0\lto \mathcal O_{x_1}\lto Q(E)\lto \mathcal O_{x_2}\lto 0\quad (x_1,x_2\in A). 
\end{gather*}
As WIT of index 2 holds for $L_i$ and $\mathcal O_{x_i}$ with respect to 
$\Phi ^D$, $E^{**}$ and $Q(E)$ satisfy WIT for index 2, and 
$(E^{**})\sphat=H^2(\Phi ^D(E^{**}))$ and 
$Q(E)\sphat=H^2(\Phi ^D(Q(E)))$ are extensions of 
skyscraper sheaves and line bundles on $\hat A$, respectively (Lemma \ref{FMlemma} (i)). 
By the semistability of $E$ and Serre duality, again, 
$\Ext ^2 (E\otimes \mathcal P_y,\mathcal O_{A})=0$ for every $y\in \hat A$, so that 
$H^2(\Phi ^D(E))=0$. By the long exact sequence of cohomology, we get
\begin{gather*}
H^0(\Phi ^D(E))=0,\\
0\to H^1(\Phi ^D(E))\to H^2(\Phi ^D(Q(E)))\to
H^2(\Phi ^D(E^{**}))\to 0. 
\end{gather*}
Thus, $E$ satisfies WIT for index 1 and is a kernel of a surjective morphism 
$Q(E)\sphat \to (E^{**})\sphat$. 

Now we check the (semi)-stability of $\hat E=H^1(\Phi ^D(E))$. If $\hat E$ is not semistable, 
we have a torsion free sub-sheaf $N_1$ of $\hat E$ with $p_{N_1}> p_{\hat E}$, where
$p$'s are reduced Hilbert polynomials. 
Noting that $H^2(\Phi^D(Q(E)))$ is $\mu$-semistable 
as it is an extension of numerically trivial line bundles, 
$c_1(N_1)=0$ since $N_1$ injects to $H^2(\Phi ^D(Q(E)))$. 
If $\rank N_1=2$, $\displaystyle p_{\hat E}=p_{N_1}+\frac{\length (\hat E/N_1)}2$. 
Therefore, $N_1$ cannot be destabilizing. 
If $\rank N_1=1$, we have a short exact sequence
\[
0\lto N_1\lto \hat E\lto N_2\lto 0
\]
with $N_2$ torsion free of rank 1, $c_1(N_2)=0$, and $c_2(N_1)+c_2(N_2)=2$.
If $\hat E$ is not semistable and $N_1$ destabilizing, $c_2(N_1)$ must be 0, i.e. $N_1$ is 
a line bundle. On the other hand, we have the exact sequence
\[
0=H^2(\hat \Phi^D(\hat E))\lto H^2(\hat \Phi ^D(N_1))\lto 0, 
\]
where we used the spectral sequence \eqref{FMSS} and Lemma \ref{FMlemma} (i). 
This is a contradiction. Thus, $\hat E$ is always semistable. 

$E$ is strictly semistable if and only if $E$ is an extension of rank 1 torsion free sheaves with 
$c_1=0, c_2=1$. The argument above implies that $\hat E$ is strictly semistable if $E$ is, and vice versa. 
\end{proof}

\begin{proof}[Proof of Theorem \ref{FM}]
(i) and (iii) are immediate consequences of the proposition above. 
The commutativity of \eqref{FMBgamma} also follow from the proof of 
the proposition. The proof of (ii) is the same as in \cite{Y}, \S4. 
\end{proof}

\section{GIT description of $B_{\gamma}$}

In this section, we give a GIT description of the fiber $B_{\gamma}$ of 
the morphism $\varphi$ to the Donaldson-Uhlenbeck compactification. 
Let us begin with fixing our notations. 

\begin{definition}\label{GITmodel}
Let us identify $\gamma=([L],[p])\in (\hat A/\{\pm 1\}) \times (A/\{\pm 1\})$ with a pair of $0$-cycles 
\[
\gamma =(\gamma _{\hat A}=\sum _{[L]\in \hat A} n_L [L],
\gamma _A=\sum _{p\in A} n_p p)\in \Sym ^2(\hat A)\times \Sym ^2(A),
\]
We define sheaves of $\mathbb C$-vector spaces $\mathcal V_{\gamma _{\hat A}}$ and 
$\mathcal Q_{\gamma _A}$ of finite length on $\hat A$ and $A$, respectively, by the stalks 
\[
\mathcal V_{\gamma _{\hat A, [L]}} = \mathbb C^{n_L},\quad 
\mathcal Q_{\gamma _A, p}=\mathbb C^{n_p}
\]
where $\mathbb C^0=0$ by convention. Note that $\dim \Gamma (\mathcal V_{\gamma _{\hat A}})=
\dim \Gamma (\mathcal Q_{\gamma_A})=2$. Using the notation
\[
N(V)=\{(A_1,A_2)\in \mathfrak sl(V)^{\oplus 2}\mid [A_1,A_2]=O, \; A_1^{i_1}A_2^{i_2}=O 
\;(i_1+i_2=\dim V)\}, 
\]
we define an affine scheme $Y_{\gamma}$ by
\[
Y_{\gamma} = N(\mathcal V_{\gamma _{\hat A}})\times 
\Hom _{\mathbb C} (\Gamma(\mathcal V_{\gamma _{\hat A}}),\Gamma (\mathcal Q_{\gamma _A}))
\times N(\mathcal Q_{\gamma _A})
\]
and a group $G_{\gamma}$ by
\[
G_{\gamma} = \Aut (\mathcal V_{\gamma _{\hat A}})\times \Aut (\mathcal Q_{\gamma _A}).
\]
Note that $\Aut (\mathcal V_{\gamma _{\hat A}})$ 
(resp. $\Aut (\mathcal Q_{\gamma _A})$) acts on $N(\mathcal V_{\gamma _{\hat A}})$ 
(resp. $N (\mathcal Q_{\gamma _A})$) by adjoint. 
We can regard $G_\gamma$ as a subgroup of 
\[
GL (\Gamma (\mathcal V_{\gamma _{\hat A}}))
\times GL (\Gamma (\mathcal Q_{\gamma _A}))\cong GL(\mathbb C^2)\times GL(\mathbb C^2).
\] 
We define a character $\chi : GL (\Gamma (\mathcal V_{\gamma _{\hat A}}))\times 
GL (\Gamma (\mathcal Q_{\gamma _A}))\to \mathbb C^*$ by 
\[
\chi=(\det {}_{\Gamma(\mathcal V_{\gamma _{\hat A}})})^{-1}\cdot 
(\det {}_{\Gamma(\mathcal Q_{\gamma _A})})
\]
and define $\chi _{\gamma}: G_{\gamma}\to \mathbb C^*$ as the composition
\[
\chi _{\gamma}: 
G_{\gamma}\hookrightarrow GL (\Gamma (\mathcal V_{\gamma _{\hat A}}))\times 
GL (\Gamma (\mathcal Q_{\gamma _A}))\overset{\chi}{\to} \mathbb C^*.
\]
\end{definition}

The theoretic basis of the proof of Theorem \ref{strB} is the following theorem. 

\begin{theorem}\label{GITdescr}
Under the same notation as in Theorem \ref{strB}, we have an isomorphism
\[
B_{\gamma} \cong Y_{\gamma} \underset{\chi _{\gamma}}{\git} G_{\gamma}
=\Proj \left( \bigoplus _{n=0}^{\infty} A(Y_{\gamma})^{G_{\gamma}, \chi _{\gamma}^n} \right), 
\]
where $A(Y_{\gamma})$ is the affine coordinate ring of $Y_{\gamma}$ and 
$A(Y_{\gamma})^{G_{\gamma}, \chi _{\gamma}^n}$ is the vector space of $G_{\gamma}$-semi-invariants 
whose character is $\chi _{\gamma}^n$. 
\end{theorem}

To prove the theorem, we have to establish a relationship between the points in $B_{\gamma}$ 
and points in $Y_{\gamma}$. 
For that purpose, we need the following 

\begin{lemma}\label{extension}
Notation as above. $N(\mathcal Q_{\gamma _A})$ parametrizes the artinian 
$\mathcal O_A$-module structures on $\mathcal Q_{\gamma _A}$
up to the conjugation of $\Aut (\mathcal Q_{\gamma _A})$. Similarly, 
$N(\mathcal V_{\gamma _{\hat A}})$ parametrizes the (possibly trivial) extension 
data 
\[
0\lto L \lto F\lto L^{-1}\lto 0
\]
up to the conjugation of $\Aut (\mathcal V_{\gamma _{\hat A}})$. 
\end{lemma}

\begin{proof}
The former assertion is clear. 
The latter is just a consequence of Lemma \ref{FMlemma} (i): 
we can write $H^2(\Phi ^D(L))=\mathcal O_y$ for some $y\in \hat A$. 
Artinian $\mathcal O_{\hat A}$-module structure on $\mathcal O_y\oplus \mathcal O_{-y}$ 
is parametrized by $N(\mathcal O_y\oplus \mathcal O_{-y})$ 
up to the conjugation by $\Aut (\mathcal O_y\oplus \mathcal O_{-y})$, 
where we naturally identify $\mathcal V_{\gamma _{\hat A}}$ with 
$\mathcal O_y\oplus \mathcal O_{-y}$. 
\end{proof}

\begin{remark}\label{extension-rem}
Let 
\[
0\lto L_1\lto F\lto L_2\lto 0
\]
be a non-trivial extension with $L_1=L_2=L\in \Pic ^0(A)$. Applying 
$H^2(\Phi ^D(-))$, we get 
\[
0\lto \mathcal O_{y_2}\lto \mathcal O_Z\lto \mathcal O_{y_1}\lto 0, 
\]
where $\mathcal O_{y_i}=H^2(\Phi ^D(L_i))=\mathcal O_y$. $Z$ is 
a length $2$ subscheme on $\hat A$ concentrated at $y$. 
If we identify $\mathcal O_Z$ with $\mathcal V_{\gamma _{\hat A}}$, 
one has $(B_1,B_2)\in N(\mathcal V_{\gamma _{\hat A}})$ corresponding to 
the scheme structure on $Z$. 
The one dimensional subspace of $\mathcal V_{\gamma _{\hat A}}$ that is annihilated by 
$B_1$ and $B_2$ corresponds to the sheaf $\mathcal O_{y_1}$, 
and accordingly to the only numerically trivial 
sub-line bundle $L_1\hookrightarrow F$. 
\end{remark}

\subsection{}\label{pullbackext}
Let us take $[E]\in B_{\gamma}$ with $\gamma=([L],[p])$. 
Then, $E$ fits into a short exact sequence
\[
0\lto E\lto F=E^{**}\overset{\overline{\Psi}}{\lto} Q(E)\lto 0. 
\]
Obviously $Q(E)\cong \mathcal Q_{\gamma _A}$ as sheaf of $\mathbb C$-vector spaces. 
Let $\iota : \Supp (Q(E))\to A$ be the inclusion. Then, 
$\overline{\Psi}$ is in one to one correspondence with 
\[
\psi : \iota ^{-1}(F) \to Q(E),
\]
which corresponds furthermore to an element 
\[
\psi \in \Hom (\Gamma(\mathcal V_{\gamma _{\hat A}}),\Gamma(\mathcal Q_{\gamma _A}))
\] 
up to a choice of isomorphisms 
$\iota ^{-1}(F)\cong \iota ^{-1}(\Gamma(\mathcal V_{\gamma _{\hat A}})\otimes \mathcal O_A)$  
and $Q(E)\cong \mathcal Q_{\gamma _A}$. 

But $\psi$ disregards the $\mathcal O_A$-module structure on $\mathcal Q_{\gamma _A}$ 
and the extension data
\[
0\lto L\lto F\lto L^{-1}\lto 0. 
\]
The former is described by $N(\mathcal Q_{\gamma _A})$ and the latter is also described by 
$N(\mathcal V_{\gamma _{\hat A}})$, according to Lemma \ref{extension}. 
Therefore, the morphism $\overline {\Psi}$ corresponds to 
an element $\Psi \in Y_{\gamma}$ up to the difference of $G_{\gamma}$-action. 

Now, Theorem \ref{GITdescr} is a direct consequence of the following 

\begin{proposition}
Let $\Psi\in Y_{\gamma}$ and consider the corresponding morphism of $\mathcal O_A$-modules 
$\overline \Psi:F\to \mathcal Q_{\gamma _A}$ as above. Then, the following are equivalent:
\begin{enumerate}[(i)]
\item $\overline \Psi$ is surjective and $E=\Ker \overline \Psi$ is semistable (resp. stable). 
\item $\overline \Psi$ is surjective and for every sub-line bundle $M\hookrightarrow F$ with 
$\mu (M)=\mu (F)$, $\dim (\overline \Psi(M))\geqslant 1$ (resp. $>1$). 
\item $\Psi$ is a $(G_{\gamma}, \chi)$-semistable (resp. stable) point. 
\end{enumerate}
\end{proposition}

\begin{proof}
The equivalence of (i) and (ii) is a consequence of Lemma 2.1.2 of \cite{OG6} 
(see also \cite{OG10}, Lemma 1.1.5).  The equivalence of (ii) and (iii) is 
an easy and classical application of Hilbert-Mumford's numerical criterion, 
and goes exactly in the same way as the proof 
of Proposition 2.3 of \cite{N}. The details are left to the reader. 
\end{proof}

\begin{remark}\label{lepotier}
Let $\mathcal F$ be the universal extension on $Y_{\gamma}\times A$ and
\[
\underline{\Psi} : \mathcal F\to \mbox{pr}_A^* \mathcal Q_{\gamma _A}
\]
the universal homomorphism. 
Let $[\mathcal E]=[\mbox{pr}_A^* \mathcal Q_{\gamma _A}]-[\mathcal F] \in K(Y_{\gamma}\times A)$ 
be the universal kernel of $\underline{\Psi}$ in the $K$-group. Le Potier's morphism
$\lambda _{[\mathcal E]}: K(A)\to \Pic(Y_{\gamma})$ is defined by
\[
\lambda _{[\mathcal E]}(\alpha) = \det ((\mbox{pr}_{Y_{\gamma}})_! ([\mathcal E]\otimes (\mbox{pr}_A)^*(\alpha)). 
\]
The character $\chi _{\gamma}$ is nothing but the character of the line bundle 
$\lambda _{[\mathcal E]}([\mathcal O_A]+[\mathcal O_q])$, where $q$ is a point on $A$. 
This determinant line bundle gives the relatively ample divisor for $M\to M^{DU}$ 
(see \cite{P}, \S 7, see also \cite{H-L},  Chap. 8), therefore, 
$\chi _{\gamma}$ is the only natural choice of polarization 
to describe $B_{\gamma}$ as a GIT quotient of $Y_{\gamma}$. 
\end{remark}

\section{Calculation of the invariant rings}

In this section, we actually calculate the homogeneous invariant ring
\[
\mathscr R_{\gamma}
=\bigoplus _{n=0}^{\infty} A(Y_{\gamma})^{G_{\gamma}, \chi _{\gamma}^n}
\]
appeared in Theorem \ref{GITdescr} and complete the proof of Theorem \ref{strB}. 
The method is completely the same as in \cite{N}, \S 3. 
The calculation itself is also quite parallel to the calculation of \emph{op. cit.}, especially 
\S 3.2 and \S 3.3. 
The reader will find a little bit more detailed explanation there. 

\subsection{}
First, we consider the case in which neither $L$ nor $p$ is 2-torsion 
for $\gamma =([L], [p])$. According to Definition \ref{GITmodel} and Theorem 
\ref{GITdescr}, $B_{\gamma}\cong Y_{\gamma} \underset{\chi _{\gamma}}{\git} G_{\gamma}$ 
for 
\[
\begin{aligned}
Y_{\gamma}& = \Hom (\mathbb C^2,\mathbb C^2),\\
G_{\gamma} &=\mathbb C^*\times \mathbb C^*\times \mathbb C^*\times \mathbb C^* ,\\
\chi _{\gamma}
&= \frac1{\id_{\mathbb C^*}}\cdot \frac1{\id_{\mathbb C^*}}\cdot \id_{\mathbb C^*}\cdot \id_{\mathbb C^*}. 
\end{aligned}
\]
We write $\Psi\in Y_{\gamma}$ as
\[
\Psi =\begin{pmatrix} z_{11} & z_{12} \\ z_{21} & z_{22}\end{pmatrix}. 
\]
$G_{\gamma}$ acts on $Y _{\gamma}$ by
\[
g\Psi = 
\begin{pmatrix}
t_1^{-1}s_1z_{11} & t_2^{-1}s_1z_{12}\\
t_1^{-1}s_2z_{21} & t_2^{-1}s_2z_{22}
\end{pmatrix}
\qquad (g=(t_1,t_2,s_1,s_2)\in G_{\gamma}). 
\]
It is immediate to see that the ring $\mathscr R_{\gamma}$ 
of $(G_{\gamma},\chi _{\gamma})$-semi-invariants
is the polynomial ring generated by 
\[
\xi _1=z_{11}z_{22},\quad \xi _2=z_{12}z_{21}. 
\]
This means that $B_{\gamma}=\Proj \mathscr R_{\gamma}=\mathbb P^1$. 
$E=\Ker (\overline\Psi : L\oplus L^{-1}\to \mathcal Q_{\gamma _A})$ 
is strictly semistable if and only if one of the entries of $\Psi$ vanishes, i.e., 
$\xi_1\xi_2=0$. This shows that $(B_{\gamma}\cap \Sigma)_{red}$ consists 
of two points. 

\subsection{}\label{2torsionp}
Let us assume $p$ is 2-torsion, but $L$ is not, i.e., $p=-p$ and $L\not\cong L^{-1}$. 
Then, our GIT setting is given by
\[
\begin{aligned}
Y_{\gamma} &= \Hom (\mathbb C^2,Q)\times N(Q)\qquad (Q=\mathbb C^2), \\
G_{\gamma} &= \mathbb C^*\times \mathbb C^*\times GL(Q), \\
\chi _{\gamma} &=\frac1{\id _{\mathbb C^*}}\cdot \frac1{\id _{\mathbb C^*}}\cdot {\det }_Q.
\end{aligned}
\]

The generating set of the ring of $SL(Q)$-invariants is given by the 
\emph{symbolic method} of classical invariant theory (see, for example, \cite{PV}, Theorem 9.3). 
Writing $\Psi\in Y_{\gamma}$ by coordinates as
\begin{multline*}
\Psi = ( (v_1=\begin{pmatrix} z_{11} \\ z_{21}\end{pmatrix}, 
v_2=\begin{pmatrix} z_{21} \\ z_{22}\end{pmatrix}) ; (A_1,A_2))
\in \Hom (\mathbb C^2, Q)\times N(Q)\\\cong (Q\oplus Q)\times N(Q), 
\end{multline*}
the invariant ring $A(Y_{\gamma})^{SL(Q)}$ is generated by
\[
\begin{aligned}
\xi _0 &= \det (v_1 \mid v_2)\\
\xi _1 &=\det (A_1v_1\mid v_1), &\xi _2&=\det (A_2v_1\mid v_1),\\
\xi _3 &=\det (A_1v_1\mid v_2), &\xi _4&=\det (A_2v_1\mid v_2),\\
\xi _5 &=\det (A_1v_2\mid v_2), &\xi _6&=\det (A_2v_2\mid v_2), 
\end{aligned}
\]
taking the fact into account that $A_1$ and $A_2$ commute each other 
and satisfy the relation $A_1^2=A_1A_2=A_2^2=0$ (see, \cite{N} \S 3.3). 
A Gr\"obner basis calculation using a computer algebra system (the author 
relies on {\sc Singular} \cite{S}) shows that the relations between these $\xi$'s 
are generated by
\begin{gather*}
\xi _1\xi_4-\xi_2\xi_3,\quad \xi_3\xi_6-\xi_4\xi_5,\\
\xi_1\xi_6-\xi_3\xi_4,\quad \xi_1\xi_6-\xi_2\xi_5,\\
\xi_3^2-\xi_1\xi_5,\quad \xi_4^2-\xi_2\xi_6.
\end{gather*}
Now we check the weights of $\xi$'s with respect to the characters, which are 
given in the following table,
\begin{center}
\begin{tabular}{c||c|c|c}
 & $\mathbb C^*\; (v_1)$ & $\mathbb C^*\; (v_2)$ & $\det {}_Q$ \\ \hline
 $\xi _0$ & -1 & -1 & 1\\
 $\xi _1$ & -2 & 0 & 1\\
 $\xi _2$ & -2 & 0 & 1\\
 $\xi _3$ & -1 & -1 & 1\\
 $\xi _4$ & -1 & -1 & 1\\
 $\xi _5$ & 0 & -2 & 1\\
 $\xi _6$ & 0 & -2 & 1\\\hline
 $\chi_{\gamma}$ & -1 & -1 & 1 
\end{tabular}
\end{center}
Therefore, the homogeneous $(G_{\gamma},\chi _{\gamma})$-invariant ring 
$\mathscr R_{\gamma}$ is generated by the $\chi_{\gamma}$-degree 1 invariants
$\xi_0, \xi_3, \xi_4$ and the $\chi_{\gamma}$-degree 2 invariants
\[
\xi_1\xi_5,\; \xi_1\xi_6,\; \xi_2\xi_5,\; \xi_2\xi_6. 
\]
But looking at the relations given before, 
these degree 2 invariants can be written as a polynomial of $\xi_3$ and $\xi_4$. 
This means that $\mathscr R_{\gamma}=\mathbb C[\xi_0,\xi_3,\xi_4]$, so that 
$B_{\gamma}\cong \mathbb P^2$. 

Taking an appropriate coordinate of $Q$, namely, replacing $\Psi$ by 
another appropriate point in the $G_{\gamma}$-orbit, 
we may assume
\[
A_1=\begin{pmatrix} 0&0 \\ a_1&0\end{pmatrix},\quad 
A_2=\begin{pmatrix} 0&0 \\ a_2&0\end{pmatrix}.  
\]
Then, $\xi_0, \xi_3, \xi_4$ are written as
\[
\xi_0=z_{11}z_{22}-z_{21}z_{12},\quad 
\xi _3=-a_1z_{11}z_{12},\quad 
\xi _4=-a_2z_{11}z_{12}. 
\]
In this coordinate of $Q$, $E=\Ker (\overline\Psi : L\oplus L^{-1}\to \mathcal Q_{\gamma _A})$ 
is strictly semistable if and only if $a_1=a_2=0$ or $z_{11}z_{12}=0$, i.e., 
$\xi_3=\xi_4=0$. Thus, $(B_{\gamma}\cap \Sigma)_{red}$ is one point set. 

\subsection{}
Let us assume $L$ is 2-torsion, but $p$ is not. Then, 
\[
\begin{aligned}
Y_{\gamma} &=N(V)\times \Hom(V, \mathbb C^2)\qquad (V\cong \mathbb C^2),\\
G_\gamma &=GL(V)\times \mathbb C^*\times \mathbb C^*, \\
\chi _{\gamma} &= ({\det}_V)^{-1}\cdot \id_{\mathbb C^2}\cdot \id_{\mathbb C^2}. 
\end{aligned}
\]
The situation is exactly ``dual'' to the situation in \S \ref{2torsionp}. 
Therefore, the calculation of the invariant ring goes in 
completely the same way (except that we have to transpose every matrix appeared) 
and we conclude that $B_{\gamma}\cong \mathbb P^2$, 
also in this case. Or, by Theorem \ref{FM}, this case can be simply reduced  
to \S \ref{2torsionp}. 

\subsection{}
Finally, we consider the case where both of $L$ and $p$ are 2-torsion. $Y_{\gamma}$, 
$G_{\gamma}$, and $\chi _{\gamma}$ corresponding to $\gamma = ([L],[p])$ are given by
\[
\begin{aligned}
Y_{\gamma} &= N(V)\times \Hom (V, Q)\times N(Q)\qquad (V=\mathbb C^2, \; Q=\mathbb C^2),\\
G_{\gamma} &= GL(V)\times GL(Q),\\
\chi _{\gamma} &= (\det {}_{V})^{-1}\cdot (\det {}_{Q}). 
\end{aligned}
\]
We write 
\begin{multline*}
\Psi = ((B_1,B_2), (v_1=\begin{pmatrix} z_{11} \\ z_{21}\end{pmatrix}, 
v_2=\begin{pmatrix} z_{12} \\ z_{22} \end{pmatrix} ), (A_1,A_2) )\\
\in N(V)\times \Hom (V,Q)\times N(Q). 
\end{multline*}
As before, the ring of $SL(Q)$-invariants is generated by 
\[
\begin{aligned}
\xi _0 &= \det (v_1 \mid v_2)\\
\xi _1 &=\det (A_1v_1\mid v_1), &\xi _2&=\det (A_2v_1\mid v_1),\\
\xi _3 &=\det (A_1v_1\mid v_2), &\xi _4&=\det (A_2v_1\mid v_2),\\
\xi _5 &=\det (A_1v_2\mid v_2), &\xi _6&=\det (A_2v_2\mid v_2), 
\end{aligned}
\]
plus the entries of $B_1$ and $B_2$. Note that $\xi _0$ is also a $SL(V)$-invariant. 
$GL(V)$ acts on $\xi_1 ,\cdots , \xi _6$ through the adjoint action on 
\[
X_1=\begin{pmatrix} 
\xi _3 & \xi _5 \\ 
-\xi _1 & -\xi _3
\end{pmatrix},\quad 
X_2=\begin{pmatrix}
\xi_4 & \xi _6\\
-\xi _2 & -\xi _4
\end{pmatrix} .
\]
The symbolic method tells us that the ring of $SL(V)$-invariants with respect to 
$B_1, B_2$ and $\xi_i$'s are given by 
\[
\begin{aligned}
\zeta _0&=\xi _0,\\
\zeta _1&=\tr (B_1X_1), & \zeta_2&=\tr (B_2X_1),\\
\zeta _3&=\tr (B_1X_2), & \zeta_4&=\tr (B_2X_2),
\end{aligned}
\]
subject to the only relation $\zeta_1\zeta_4-\zeta_2 \zeta_3=0$ (see \cite{N}, \S 3.3). 
The weights for $\zeta _i$'s are all the same as $\chi _{\gamma}$. Therefore, 
\[
\mathscr R_{\gamma}=\mathbb C[\zeta _0, \cdots ,\zeta _4] / (\zeta_1\zeta_4-\zeta_2 \zeta_3), 
\]
which means that $B_{\gamma}=\Proj \mathscr R_{\gamma}$ is a cone over a quadric surface in 
$\mathbb P^4$. 

We pass to a point $\Psi$ with
\[
A_1=\begin{pmatrix} 0&0 \\ a_1&0\end{pmatrix},\quad 
A_2=\begin{pmatrix} 0&0 \\ a_2&0\end{pmatrix},\quad  
B_1=\begin{pmatrix} 0&b_1 \\ 0&0\end{pmatrix},\quad 
B_2=\begin{pmatrix} 0&b_2 \\ 0&0\end{pmatrix},  
\]
by $G_{\gamma}$-action. On such a point, $\zeta _i$'s are written as
\[
\begin{aligned}
\zeta_0 & = z_{11}z_{22}-z_{21}z_{12}, \\
\zeta _1 &= a_1b_1 z_{11},\quad
\zeta _2 = a_1b_2 z_{11},\quad
\zeta _3 =a_2 b_1 z_{11},\quad 
\zeta _4 = a_2 b_2 z_{11}.
\end{aligned}
\]
Noting that 
$\begin{pmatrix} 1 \\ 0\end{pmatrix} \in \Gamma (\mathcal V_{\gamma _{\hat A}})$
corresponds to the only numerically trivial sub-line bundle $L\hookrightarrow F$ 
if $F$ is non-splitting (Remark \ref{extension-rem}), 
it is immediate to see that $E=\Ker \overline \Psi$ is strictly semistable 
if and only if $a_1=a_2=0$, or $b_1=b_2=0$, or $z_{11}=0$. This implies that 
$(B_{\gamma}\cap \Sigma)_{red}$ is defined by $\zeta_1=\zeta _2=\zeta _3=\zeta _4=0$, 
namely the vertex of the cone $B_{\gamma}$. 

This completes the proof of Theorem \ref{strB}. 

\section{Local equations via deformation theory}

In this section, we determine the scheme structure of $B_{\gamma}\cap \Sigma$ using 
deformation theory.  The whole section will be spend for the proof of the following 

\begin{theorem}\label{centerofblowingup}
Notation as in \S 1. Let $J$ be the ideal of $B_{\gamma}\cap \Sigma$ in $\mathcal O_{B_{\gamma}}$. 
If neither $L$ nor $p$ is 2-torsion, $\Supp (\mathcal O_{B_\gamma} / J)$ consists of exactly two points 
and $J$ is the sum of maximal ideals corresponding to these points. 
Otherwise, $\Supp (\mathcal O_{B_\gamma}/J)$ is just a point, say $b$, and $J$ is the 
square of the maximal ideal at the point, namely, $J=\mathfrak m_b^2$.  
\end{theorem}

\subsection{}\label{deformation}
To prove the theorem, we need some preparation on deformation theory. 
Let $E$ be a semistable sheaf on a projective variety, 
$G(E)=\Aut (E)/\mathbb C^*$, and
\[
\mathscr D_E : (Art/ \mathbb C)\to (Sets)
\]
be the deformation functor of $E$, 
where $(Art/\mathbb C)$ stands for the category of artinian local $\mathbb C$-algebras.  
Moreover, assume that $G(E)$ is reductive. 
Then, Luna's \'etale slice theorem 
 implies that the functor $\mathscr D_E$ has a versal deformation space 
$0\in \Def(E)$ given by a germ of affine scheme such that 
\begin{equation}\label{versal2tangent}
(0\in \Def (E))\git G(E)\cong ([E]\in \overline M(E)),
\end{equation}
where $\overline M(E)$ is the moduli space of semistable sheaves 
that $E$ belongs to (see \cite{OG10}, Proposition 1.2.3). 
In particular, every point of $\Def (E)$ corresponds to a semistable sheaf from $\overline M(E)$. 

Now take $E=\mathfrak m_{p_1}L_1\oplus \mathfrak m_{p_2}L_2$ 
a strictly semistable sheaf of our moduli space $M$. 
For an artinian local ring $R$, we have a family of semistable sheaves $\mathcal E_R\in \mathscr D_E(R)$ 
flat over $R$. We define $\mathcal E_R^{**}$ to be the double dual of $\mathcal E_R$ on $A\times \Spec (R)$ and 
$Q(\mathcal E_R)=\mathcal E_R^{**}/\mathcal E_R$. We define a subfunctor $\mathscr D_{B,E}$ by
\[
\mathscr D_{B,E}(R)=\{\mathcal E_R\in \mathscr D_E(R) \mid Q(\mathcal E_R) \mbox{ is flat over } R\}.
\]
This is a closed subfunctor because flatness is locally closed condition and its versal deformation 
space is identified with a closed subscheme $\Def _B(E)\subset \Def (E)$. 
Furthermore, we define a closed subfunctor 
$\mathscr D_{\Sigma, E}$ by
\[
\mathscr D_{\Sigma, E}(R)=
\left\{\mathcal E_R\in \mathscr D_{B,E}(R) \, \Bigg |\;  
\begin{matrix}
\exists \mathcal L\in \Pic ^0(A\times \Spec (R)),\\
\exists \mbox{ a section } S\subset A\times \Spec (R),\\
p\in S \mbox{ and } I_S\mathcal L\hookrightarrow \mathcal E_R
\end{matrix}
\right\}.
\]
The associated versal space is a subscheme $\Def_{\Sigma}(E)\subset \Def_B(E)$. 

We have the short exact sequence
\[
0\lto E\lto F=E^{**}\overset{\overline \Psi}{\lto} Q(E)\lto 0
\]
with $F\cong L_1\oplus L_2$ and $Q(E)\cong \mathcal O_{p_1}\oplus \mathcal O_{p_2}$. 
$\overline \Psi$ is determined only at the stalks on the support of $Q(E)$. 
So we can identify $\overline \Psi$ with a morphism $\overline \Psi : \mathcal O_A^{\oplus 2}\to Q(E)$. 
Let $\mathscr D_{\overline \Psi}$ be the deformation functor of $\overline\Psi$, namely
\[
\mathscr D_{\overline \Psi}(R)
=\left\{ \overline\Psi _R:\mathcal O_{A}^{\oplus 2}\otimes R\to \mathcal Q_R \; \bigg|\; 
\begin{matrix}
 \overline\Psi _R\mbox{ surjective, } \mathcal Q_R \mbox{ flat over $R$ and}\\
 \overline\Psi _R\otimes (R/\mathfrak m_R)\cong \overline\Psi
\end{matrix}
\right\}
\]
The versal space $\Def (\overline \Psi)$ to the functor $\mathscr D_{\overline \Psi}$ is given by 
an affine neighborhood of $\Quot (\mathcal O _A^{\oplus 2}, 2)$ at $\overline \Psi$. 

\begin{proposition}\label{defdecomp}
Let $E=\mathfrak m_{p_1}L_1\oplus \mathfrak m_{p_2}L_2$ 
and $\overline \Psi :E^{**}\to Q(E)$ as above. 
The functor $\mathscr D_{B, E}$ is isomorphic to the product 
$\mathscr D_F\times \mathscr D_{\overline \Psi}$. 
\end{proposition}

\begin{proof} (cf. Lemma 9.6.1 of \cite{H-L})
Take $\mathcal E_R\in \mathscr D_{B,E}(R)$ and consider a locally free resolution
\[
0\lto \mathcal F_1\lto \mathcal F_2\lto \mathcal E_R\lto 0
\]
(note that the homological dimension of $E$ is $1$). By dualizing the sequence, we get
\[
0\lto \mathcal E_R^*\lto \mathcal F_0^*\lto \mathcal F_1^*\lto 
\mathcal Ext^1 _{\mathcal O_A\otimes R}(\mathcal E_R,\mathcal O_A\otimes R)\lto 0. 
\]
The local duality theorem implies that
\[
(\mathcal Ext^1 _{\mathcal O_A\otimes R}(\mathcal E_R,\mathcal O_A\otimes R)_p)\sphat
\cong \Hom _{\widehat {\mathcal O_A\otimes R}}(H^0_{\mathfrak m_p} (Q(\mathcal E_R)),
E(\widehat{\mathcal O_A\otimes R}/\mathfrak m_p))
\]
where $p$ is any closed point on $A\times \Spec (R)$
and $\sphat\, $ denote the completion at the maximal ideal $\mathfrak m_p$. 
$Q(\mathcal E_R)$ is locally free $R$-module since
it is $R$-flat. Therefore,  $\mathcal Ext^1 _{\mathcal O_A\otimes R}(\mathcal E_R,\mathcal O_A\otimes R)$ 
is flat over $R$ and so is $\mathcal E_R^*$. This shows that the formation of the dual of $\mathcal E_R$ commute 
with base change and we can say the same thing for the operation of taking double dual. 
Therefore, the correspondence $\mathcal E_R\mapsto 
(\mathcal E_R^{**}, \mathcal E_R^{**}\to Q(\mathcal E_R))$ 
defines a natural transformation 
$\delta :\mathscr D_{B,E}\to \mathscr D_F\times \mathscr D_{\overline \Psi}$. 

Conversely, assume that we are given  
$\mathcal F_R\in \mathscr D_F(R)$ and 
$(\overline \Psi _R: \mathcal O_A^{\oplus 2}\otimes R \to \mathcal Q_R)\in \mathscr D_{\overline \Psi}(R)$. 
Let $\iota :\Supp (Q(E))\hookrightarrow A$ be the natural inclusion. Then, $\overline \Psi _R$ can be identified with 
a surjective homomorphism $\overline \Psi _R: \iota ^{-1}(\mathcal O_{A}^{\oplus 2}\otimes R)\to \mathcal Q_R$ by 
the previous lemma. Fixing an isomorphism 
$\iota ^{-1}(L_1\oplus L_2)\cong \mathcal \iota ^{-1}(\mathcal O_{A}^{\oplus 2})$ once for all, 
$\mathcal F_R$ and $\overline \Psi _R$ gives a surjective morphism $\mathcal F_R\to \mathcal Q_R$, and 
its kernel $\mathcal E_R$ is an element of $\mathscr D_{B,E}(R)$. This correspondence gives the inverse of 
$\delta$. 
\end{proof}

\begin{lemma}\label{FMloc}
Let $F=L_1\oplus L_2$ with $L_i\in \Pic ^0(A)$ and consider its Fourier-Mukai transform
$\hat F=H^2(\Phi ^D(F))=\mathcal O_{y_1}\oplus \mathcal O_{y_2}$ $(y_1,y_2\in \hat A)$. 
Then $\Psi ^D$ induces an isomorphism between deformation spaces
\[
\Def (F)\overset{\sim}{\to} \Def (\hat F). 
\]
In particular, every point in $\Def (F)$ is identified with an extension of numerically trivial line bundles 
on $A$. 
\end{lemma}

\begin{proof}
Taking Lemma \ref{FMlemma} (i) into account, 
it is sufficient just to apply Theorem 1.6 of \cite{Muk2}. 
\end{proof}

\subsection{}\label{def_fiber}
We have a natural cycle map $c:\Def (\overline \Psi)\to \Sym ^2(A)$ 
by sending $(\overline \Psi: \mathcal O_A^{\oplus 2}\otimes R \to \mathcal Q_R)$ 
to the family of $0$-cycles associated with $\mathcal Q_R$. 
Similarly, according to Lemma \ref{FMloc}, we have a classifying morphism 
$\gr:\Def (F)\to \Sym ^2(\hat A)$. 
We define 
\[
\Def (\overline \Psi)_{\gamma} =(c^{-1}(c(\overline \Psi)))_{red}, \quad 
\Def (F)_{\gamma} =(\gr^{-1}(\gr(F)))_{red}. 
\]
Using the isomorphism given in Proposition \ref{defdecomp}, we get a morphism 
\[
\varphi_{loc}=(\gr, c):\Def _B (E)\to \Sym ^2 (\hat A)\times \Sym ^2(A), 
\]
which is a deformation space analog of the morphism 
\[
\varphi _{|B_M}: B_M\to \Sym ^2(\hat A)\times \Sym ^2 (A)
\]
appeared in \S \ref{descrvarphi}. 
We denote by $\Def _B(E)_{\gamma}$ the reduction of the fiber of $\varphi _{loc}$ 
over $\gamma=\varphi([E])$. Obviously, 
\begin{equation}\label{defdecomp2}
\Def _B(E)_{\gamma}\cong \Def (F)_{\gamma}\times 
\Def (\overline \Psi)_{\gamma}.
\end{equation}

As we have the isomorphism \eqref{versal2tangent} in \S \ref{deformation}, 
we get the following 

\begin{proposition}\label{localeqviadef}
The germ $(0\in \Def _B(E)_{\gamma}) \git G(E)$ is isomorphic to $([E]\in B_{\gamma})$. 
\end{proposition}

\subsection{}\label{lochalf}
Now, we proceed to the proof of Theorem \ref{centerofblowingup}. 
Let us first consider the case in which 
$E=\mathfrak m_pL\otimes \mathfrak m_pL^{-1}$ with $L\not\cong L^{-1}$ and $p\in A$ a 2-torsion point, 
since the calculation in this case explains well the idea used also in the other cases, although it is not the simplest case. 
We have $F=E^{**}=L\oplus L^{-1}$ and $Q(E)=\mathcal O_p^{\oplus 2}$. 
The Fourier-Mukai transform $\hat F=H^2(\Phi ^D(F))$ is a direct sum of the structure sheaves at two different 
points on $\hat A$. Lemma \ref{FMloc} infers that $\Def (F)\cong \mathbb C^4$ and $\Def(F)_{\gamma}$ is 
a reduced point. Thus, $\Def _B(E)_\gamma\cong \Def (\overline \Psi)_{\gamma}$ by \eqref{defdecomp2}. 
Therefore, the deformation space is completely local in nature and  
can be calculated as in the following without any calculation of higher obstruction. 

The Zariski tangent space to the functor $\mathscr D_{\overline\Psi}$ is 
$V:=\Hom (\mathfrak m_p\oplus \mathfrak m_p, \mathcal O_p\oplus \mathcal O_p)\cong \mathbb C^8$. 
This means that we can regard $\Def (\overline\Psi)$ as a germ of a closed subscheme 
in $\Spec (\mathbb C[V^*])$ at the origin. 
Fixing a coordinate $\mathcal O_{A,p}\cong \mathbb C[x,y]_{(x,y)}$ 
at $p$, $\overline \Psi$ is presented by
\[
\mathcal O^{\oplus 4}\overset{P}{\lto} \mathcal O^{\oplus 2}\overset{\overline\Psi}{\lto} Q\lto 0
\]
with 
\[
P=
\begin{pmatrix}
x & y & 0 & 0 \\
0 & 0 & x & y
\end{pmatrix}
\]
Let's take a coordinate system $z_1, \cdots ,z_8\in V^*$. 
The ``universal deformation'' of $\overline\Psi$ is described by 
\[
\mathcal O[V^*]^{\oplus 4}\overset{\widetilde P}{\lto}
\mathcal O[V^*]^{\oplus 2}\lto \widetilde Q\lto  0,
\]
where
\[
\widetilde P=
\begin{pmatrix}
x+z_1 & y+z_2 & z_3 & z_4 \\
z_5 & z_6 & x+z_7 & y+z_8
\end{pmatrix}. 
\]
Here, we omit the localization at the origin as there is no fear of confusion. 
From this presentation, we know that $\widetilde Q$ is generated as a $\mathbb C[V^*]$-module 
by $q_1$ and $q_2$ that are
the images of $\begin{pmatrix} 1\\ 0 \end{pmatrix}$ and $\begin{pmatrix} 0 \\ 1 \end{pmatrix}$, 
respectively. Namely, we have a surjective homomorphism
$\Psi ':\mathbb C[V^*]^{\oplus 2}\to \widetilde Q$. This $\Psi '$ is presented by 
a matrix $P'$ obtained by eliminating $x$ and $y$ from $\widetilde P$. More precisely, 
$P'$ is calculated in the following way: we have the relations
\begin{equation}\label{elm}
\begin{aligned}
x q_1+z_1q_1+z_5q_2&=0,\quad 
y q_1+z_2q_1+z_6q_2=0,\\ 
x q_2+z_3q_1+z_7q_2&=0,\quad 
y q_2+z_4q_1+z_8q_2=0.
\end{aligned} 
\end{equation}
We can eliminate $x$ and $y$ from these relations using 
$y(xq_1)-x(yq_1)=0,\; y(xq_2)-x(yq_2)=0$, i.e.,
\begin{gather}\label{i1}
P'_1=\begin{pmatrix}
z_3z_6-z_4z_5\\
z_2z_5+z_6z_7-z_1z_6-z_5z_8
\end{pmatrix}\, ,\,
P'_2=\begin{pmatrix}
z_1z_4+z_3z_8-z_2z_3-z_4z_7\\
z_4z_5-z_3z_6
\end{pmatrix}
\end{gather}
generates the kernel of $\Psi '$, so that we have a presentation
\[
\mathbb C[V^*]^{\oplus 2}\overset{P'}{\lto} \mathbb C[V^*]^{\oplus 2}\overset{\Psi '}{\lto} \widetilde Q\lto 0
\]
with $P'=(P'_1,P'_2)$. The deformation space $\Def (\overline \Psi)$ is the strata containing the origin 
in the flattening stratification. In our case, this is the locus where $P'$ has rank 0. 
Therefore, the defining ideal $I_1$ of $\Def (\overline \Psi)$ is the ideal generated by the entries of 
$P'$, i.e., the four polynomials appeared in \eqref{i1}. 

The subvariety $\Def(Q(E))_{\gamma}$ is the locus of $v\in V$ where the support of the fiber 
$\widetilde Q\otimes \kappa (v)$ is exactly $\{p\}$, the origin in $(x,y)$-plane. Taking it into account 
that the length of $Q$ is 2, this is given by the conditions
\[
x^iy^j\cdot q_1=0, \quad x^iy^j\cdot q_2=0 \qquad (i+j=2). 
\]
These equations can be translated into polynomial equations only in $z_i$'s, using the elimination 
relation \eqref{elm}, which give rise to an ideal $I_2$. As we put the reduced scheme structure on 
$\Def (Q(E))_{\gamma}$, it is defined by the ideal  $I=\sqrt{(I_1+I_2)}$, which is calculated as
\begin{multline}\label{idealdef}
I=(z_1+z_7, \;  z_2+z_8, \; 
z_6z_7-z_5z_8,\;  
z_4z_7-z_3z_8,\\
z_4z_6+z_8^2,\; 
z_3z_6+z_7z_8,\; 
z_4z_5+z_7z_8,\; 
z_3z_5+z_7^2).
\end{multline}

According to Proposition \ref{localeqviadef}, a local model of $B_{\gamma}$ at the point $[E]$ is given by 
$\Def(Q(E))_{\gamma} \git \mathbb C^*$. In words of rings, the pull-back of the ideal $I$
to the invariant ring $\mathbb C[V^*]$ gives local equations of $B_{\gamma}$. 
As $t\in \mathbb C^*$ acts on $\widetilde P$ by 
\[
\widetilde P\mapsto 
\begin{pmatrix}
t & 0 \\ 
0 & t^{-1}
\end{pmatrix}
\widetilde P
\begin{pmatrix}
t^{-1} I_2 & O\\
O & t I_2
\end{pmatrix}
\qquad (I_2\mbox{ is the $2\times2$ unit matrix}),
\]
the invariant ring $\mathbb C[V^*]^{\mathbb C^*}$ is generated by
\begin{gather*}
t_1=z_1, \quad t_2=z_2,\quad  t_3=z_3z_5,\quad  t_4=z_3z_6,\\
t_5=z_4z_5,\quad t_6=z_4z_6,\quad t_7=z_7, \quad t_8=z_8, 
\end{gather*}
and the pull-back $\rho^{-1}(I)$ by $\rho:\mathbb C[t_1,\cdots ,t_8]\to \mathbb C[V^*]$ is 
\[
(t_1+t_7,\; t_2+t_8,\; t_3+t_7^2,\; 
t_4-t_5,\; t_5+t_7t_8,\; t_6+t_8^2), 
\]
which implies that 
\begin{equation}\label{2plane}
\mathcal O_{B_{\gamma},[E]}\cong \mathbb C[t_7,t_8]_{(t_7,t_8)}.
\end{equation}

The subscheme $(\Def _{\Sigma}(E))_{red}\cap \Def (\overline\Psi)$ is defined by 
$z_3=z_4=0$ or $z_5=z_6=0$, namely, by the ideal 
\[
I'=\sqrt{(I_1,z_3,z_4)\cdot (I_1,z_5,z_6)}. 
\]
Therefore, the intersection $\Def _B(E)_{\gamma}\cap (\Def _{\Sigma}(E))_{red}$
is defined by $I+I'$. The pull-back $\rho^{-1}(I+I')$, which gives the ideal of
$B_{\gamma}\cap (\Sigma)_{red}$ at $b=[\mathfrak m_pL\oplus \mathfrak m_pL^{-1}]$, 
is given by
\[
J=(t_7^2, t_7t_8, t_8^2)
\]
under the isomorphism \eqref{2plane}. This shows that $J=\mathfrak m_b^2$.

\subsection{}
Thanks to Theorem \ref{FM}, 
the case in which $E=\mathfrak m_pL\oplus \mathfrak m_{-p}L$ with $p\neq -p$ 
follows immediately from \S \ref{lochalf}.

\subsection{}
Next, take $E=\mathfrak m_{p}L\oplus \mathfrak m_{-p}L^{-1}$ where neither $p$ nor $L$ is not $2$-torsion. 
Then,  $\Def _B(E)\cong \Def (F)\times \Def (\overline \Psi)$ for $F=L\oplus L^{-1}$ and 
$\overline \Psi : \mathcal O_A^{\oplus 2}\to \mathcal O_p\oplus \mathcal O_{-p}$. 
As before, $\Def (F)_{\gamma}$ is only a reduced one point. $\overline \Psi$ is presented by
\[
\begin{aligned}
P &=
\begin{pmatrix}
x & y & 0 \\
0 & 0 & 1
\end{pmatrix}\quad \mbox{at $p$}, \\
P &=
\begin{pmatrix}
1 & 0 & 0 \\
0 & x & y
\end{pmatrix}\quad \mbox{at $(-p)$}.
\end{aligned}
\]
Since $\Hom (\mathfrak m_p\oplus \mathfrak m_{-p}, \mathcal O_p\oplus \mathcal O_{-p})$ is isomorphic to
\[
\left( \Hom (\mathfrak m_p,\mathcal O_p)\oplus \Hom (\mathcal O_A,\mathcal O_p)\right) \oplus 
\left( \Hom (\mathcal O_A,\mathcal O_{-p})\oplus \Hom (\mathfrak m_{-p},\mathcal O_{-p})\right),  
\]
the universal deformation of $\overline \Psi$ is given by
\[
\widetilde P=
\left(\begin{array}{@{\,}ccc@{\quad}|@{\quad}ccc@{\,}}
x+z_1 & y+z_2 & z_3 & 1 & 0 & 0 \\
0 & 0 & 1 & z_4 & x+z_5 & y+z_6 
\end{array}\right).
\]
From this, it is easy to see that $\Def (\overline \Psi)$ is unobstructed. 
$\Def (\overline \Psi)_{\gamma}$ is defined by the equations $z_1=z_2=z_5=z_6=0$, 
i.e., $\Def (\overline \Psi)_{\gamma}\cong \Spec \mathbb C[z_3,z_4]$. 
The group $G(E)\cong \mathbb C^*\ni t$ acts on $z_3$ and $z_4$ by 
\[
z_3\mapsto t^2\cdot z_3,\quad z_4\mapsto t^{-2}\cdot z_4
\]
as in \S \ref{lochalf}, so the invariant ring is just $\mathbb C[s]$ with $s=z_3z_4$, 
which gives the coordinate ring of the germ $[E]\in B_{\gamma}$. 
The intersection $\Def _B(E)_{\gamma}\cap (\Def _{\Sigma}(E))_{red}$ is defined by
$z_3z_4=0$, which means that the ideal of $B_{\gamma}\cap \Sigma _{red}$ is 
just the maximal ideal $\mathfrak m_{[E]}=(s)$. 

\begin{remark}
The claim in this case is nothing but Lemma 4.3.10 of \cite{OG6}. 
Our argument is more explicit and seems to be easier than the proof of \emph{op. cit.}. 
\end{remark}

\subsection{}\label{locfull}
Finally, let us consider the case $E=(\mathfrak m_pL)^{\oplus 2}$. 
Lemma \ref{FMloc} implies that $\Def _B(E)\cong \Def (F)\times \Def(\overline \Psi)$
with $F=L^{\oplus 2}$ and $\overline \Psi:\mathcal O_A^{\oplus 2}\to \mathcal O_p^{\oplus 2}$.   
By Lemma \ref{FMloc}, we have $\Def (F)\cong \Def (\hat F)$ with $\hat F$ is of the 
form $\mathcal O_{y}^{\oplus 2}$ with $y\in \hat A$.  We have 
\[
\Ext ^{i-1}(\mathfrak m_y^{\oplus 2},\mathcal O_y^{\oplus 2})
\cong \Ext ^i(\mathcal O_y^{\oplus 2},\mathcal O_y^{\oplus 2})\quad \mbox{for }i=1,2,
\]
and the isomorphisms commute with the obstruction maps (see \cite{H-L}, \S 2.A.8).
Therefore,  we have an isomorphism $\Def (\hat F)\cong 
\Def (\hat{\overline \Psi}: \mathcal O_{\hat A}^{\oplus 2}\to \mathcal O_y^{\oplus 2})$. 
Thus, the calculation of the ideal associated with $\Def _B(E)_{\gamma}$ is 
almost the same as in \S \ref{lochalf}. 
The ideal $I_1\mbox{ (resp. $I$) }\subset \mathbb C[z_1,\cdots, z_8,w_1,\cdots ,w_8]$ 
of $\Def _B(E)$ (resp. $\Def _B(E)_{\gamma}$) is generated by 
the polynomials in \eqref{i1} (resp. \eqref{idealdef}) plus the same polynomials but all the $z$'s 
replaced by $w$'s. 

The most significant difference is that we have $G(E)\cong SL(2)$ in this case. 
$T\in SL(2)$ acts on $\widetilde P$ by
\[
\widetilde P\mapsto T\widetilde P
(T^{-1}\otimes I_2)\]
In other words, $T$ acts on $(z_1, \cdots ,z_8, w_1, \cdots , w_8)$ via the adjoint action 
on
\[
Z_1=
\begin{pmatrix}
z_1 & z_3 \\
z_5 & z_7 
\end{pmatrix}
\, , \,
Z_2=
\begin{pmatrix}
z_2 & z_4 \\
z_6 & z_8 
\end{pmatrix}
\, ,\,
Z_3=
\begin{pmatrix}
w_1 & w_3 \\
w_5 & w_7 
\end{pmatrix}
\, , \,
Z_4=
\begin{pmatrix}
w_2 & w_4 \\
w_6 & w_8 
\end{pmatrix}.
\]
It is known that the invariant ring $\mathbb C[z_1,\cdots, z_8,w_1,\cdots , w_8]^{SL(2)}$ is 
generated by
\begin{gather*}
\tr(Z_i)\quad (i=1,2,3,4),\\
\tr(Z_iZ_j)\quad (1\leqslant i\leqslant j\leqslant 4),\\
\tr(Z_iZ_jZ_k)\quad (1\leqslant i\leqslant j\leqslant k\leqslant 4)
\end{gather*}
(see, for example, \cite{Kra}, \S3.3). However, in our case, the ideal $I$ contains 
$\tr(Z_i)$ and all the entries of $Z_1^2, Z_1Z_2, Z_2^2, Z_3^2,Z_3Z_4,Z_4^2$. 
Therefore, all the invariants above but
\[
t_1=\tr(Z_1Z_3)\, ,\, t_2=\tr(Z_1Z_4)\, ,\, t_3=\tr(Z_2Z_3)\, ,\, t_4=\tr(Z_2Z_4)
\] 
vanishes in $\mathbb C[z_1,\cdots ,z_8,w_1,\cdots, w_8]/I$. Therefore, we only need to consider 
\[
\rho: \mathbb C[t_1,\cdots, t_4]\to\mathbb C[z_1,\cdots ,z_8,w_1,\cdots, w_8]. 
\]
The pull-back is given by $\rho ^{-1}(I)=(t_2t_3-t_1t_4)$. 

By the same reason as before, the subscheme $(\Def _{\Sigma}(E))_{red}\cap \Def_B(E)_{\gamma}$
is defined by $I+I'$, where $I'=\sqrt{(I_1,z_3,z_4,w_3,z_4)(I_1,z_5,z_6,w_5,w_6)}$. 
One can check $\rho^{-1}(I+I')=(t_1,t_2,t_3,t_4)^2$. 
This proves $J=\mathfrak m_b^2$ for $b=[(\mathfrak m_pL)^{\oplus 2}]$ and completes 
the proof of Theorem \ref{centerofblowingup}. 

\begin{remark}
In \S\S \ref{lochalf} and \ref{locfull}, computer calculations on Gr\"obner basis will help the reader 
to be convinced the results of the calculations. The author used {\sc Singular} \cite{S} for calculations 
of radicals and pull-back of ideals. 
\end{remark}

\section{Cone of curves over the Donaldson-Uhlenbeck compactification}

Recall that Lehn--Sorger \cite{L-S} proved that O'Grady's resolution $\widetilde X\to X$ 
is nothing but the blowing-up along $\Sigma_{red}$. 
This in particular implies that the strict transform
$\widetilde B$ of $B$ on $\widetilde X$ is the blowing-up of $B$ along $B\cap \Sigma_{red}$. 
Theorems \ref{strB} and \ref{centerofblowingup} enables us to determine the geometry 
of every fiber of the composition $\widetilde B\to B\to \varphi(B)$. 
Using this information, it is quite easy to prove the following 

\begin{theorem}\label{cone}
Let $E$ be the exceptional divisor of the blowing-up $\pi :\widetilde X\to X$, 
$\delta$ the general fiber of $\pi_{|E}:E\to \Sigma$, and $\beta$ the general fiber 
of $\widetilde B\to \varphi(B)$. Then, 
the cone of curves on $\widetilde X$ over $X^{DU}$ (see, for example, \cite{K-M}, \S 3.6) is 
\[
\overline{NE}(\widetilde X/X^{DU})=\mathbb R_{\geqslant 0}[\delta] + 
\mathbb R_{\geqslant 0}[\beta]. 
\]
\end{theorem}

The assertion (i) of our Main Theorem is a direct consequence of this theorem and 
the cone-contraction theorem (Theorem 3.25 in \cite{K-M}); as $K_{\widetilde X}$ is 
trivial and $\widetilde B\cdot \beta=-2$ (see the lemma below), 
the contraction $f$ in Main Theorem 
is just the contraction of the ray $\mathbb R_{\geqslant 0}[\beta]$ 
that is negative with respect to $K_{\widetilde X}+\varepsilon \widetilde B$. 

\begin{lemma} [O'Grady, Perego]\label{num}
$E,\widetilde B, \delta, \beta$ as in the theorem above. 
\begin{enumerate}[(i)]
\item $E\cdot \delta=\widetilde B\cdot \beta=-2$, $E\cdot \beta=2$, and $\widetilde B\cdot \delta =1$. 
\item $B$ is $\mathbb Q$-Cartier and $\widetilde B\equiv \pi^*B-\frac12E$. 
\end{enumerate}
\end{lemma}

\begin{proof}
(i) is the table (7.3.5) of \cite{OG6}. (ii) follows from the proof of Theorem 9 in \cite{P}.
\end{proof}

\begin{proof}[Proof of Theorem \ref{cone}]
Take $\gamma=([L],[p])\in \varphi(B)=(\hat A/\{\pm 1\})\times (A/\{\pm 1\})\subset
\Sym ^2(\hat A)\times \Sym ^2(A)$ and let $\widetilde B_{\gamma}$ be the 
fiber of $\widetilde B\to \varphi (B)$ over $\gamma$ with the reduced structure. 
The cone of curves $\overline{NE}(\widetilde X/X^{DU})$ is generated by $[\delta]$ and
the union of the image of $\overline{NE}(\widetilde B_{\gamma})$ for all $\gamma$. 
Therefore, to prove the theorem, 
it is enough to determine $\overline{NE}(\widetilde B_{\gamma})$ for every $\gamma$. 

If $\gamma$ is generic, namely, neither $L$ nor $p$ is 2-torsion, $B_{\gamma}\cong \mathbb P^1$, 
therefore, $\widetilde B_{\gamma}\cong \mathbb P^1$, which is noting but $\beta$. 

Assume $p$ is 2-torsion but $L$ is not. Then, $B_{\gamma}\cong \mathbb P^2$ 
by Theorem \ref{strB} and the ideal of $B_{\gamma}\cap \Sigma _{red}$ is the square of 
the maximal ideal $\mathfrak m_b^2$ at a point $b\in B_{\gamma}$ by Theorem \ref{centerofblowingup}. 
Then, $\widetilde B_{\gamma}$ is nothing but $\mathbb F_1$ 
and $E_{|\widetilde B_{\gamma}}=2\sigma$ where $\sigma$ is the negative section of $\mathbb F_1$. 
Let $l$ be the ruling of $\mathbb F_1$. We can write $l=x\delta +y \beta$ as a numerical 
1-cycle in $\widetilde X$. We have $E\cdot l=(2\sigma\cdot l)_{\widetilde B_{\gamma}}=2$ 
and $\widetilde B\cdot l = (\pi ^*B-\frac12 E)\cdot l = B\cdot \pi (l)-1$ by Lemma \ref{num}. 
But, we know that $-B_{|B_{\gamma}}\equiv \mathcal O_{B_{\gamma}}(1)$ from 
Theorem 9 of \cite{P} and Remark \ref{lepotier}. As $\pi (l)$ is a line on $B_{\gamma}=\mathbb P^2$, 
we get $\widetilde B\cdot l=-2$. This implies that $x=0$ and $y=1$, i.e., $l$ is numerically 
equivalent to $\beta$. This shows that $\overline{NE}(\widetilde B_{\gamma})$ is 
spanned by $\delta\equiv \sigma$ and $\beta\equiv l$. The same argument applies 
for the case in which $L$ is 2-torsion but $p$ is not. 

Now, let us assume both of $L$ and $p$ are 2-torsion. Then, $B_{\gamma}$ is a 3-fold 
that is a cone over a smooth quadric surface $Q$ in $\mathbb P^4$ (Theorem \ref{strB}) and 
the ideal of $B_{\gamma}\cap \Sigma _{red}$ is the square of the maximal ideal 
$\mathfrak m_b^2$ at the vertex of $B_{\gamma}$ (Theorem \ref{centerofblowingup}). 
Take a plane $\Pi$ spanned by a line on $Q$ and the vertex of $B_{\gamma}$. 
The strict transform $\widetilde \Pi$ in $\widetilde B_{\gamma}$ is again $\mathbb F_1$. 
Take a ruling $l$ of $\widetilde \Pi$. Then, we conclude that $l$ is numerically equivalent to 
$\beta$ by the same argument as above applied on $\widetilde \Pi$. The planes of the form of 
$\Pi$ sweep the whole $B_{\gamma}$. Therefore, $\overline{NE}(\widetilde B_{\gamma})$ is
spanned by $\delta$ and $\beta$, also in this case. 
\end{proof}

\begin{proof}[Proof of Main Theorem]
It remains to prove (ii). For any $\gamma$, $\widetilde B_{\gamma}$ has a $\mathbb P^1$-bundle structure 
whose fiber is numerically equivalent to $\beta$. This already means that 
$f_{|\widetilde B}:\widetilde B\to Z=f(\widetilde B)$ is $\mathbb P^1$-bundle. 
Theorems 1.3 and 1.4 of \cite{W} implies that $Z=f(\widetilde B)$ is smooth symplectic
and is a locally trivial family of $A_1$-singularities. As $Z$ obviously birationally dominates 
$\varphi (B)=(\hat A/\{\pm 1\})\times (A/\{\pm 1\})$, $Z$ is a symplectic resolution of 
$(\hat A/\{\pm 1\})\times (A/\{\pm 1\})$. The remaining assertion is a consequence of the 
following easy 

\begin{claim}
Let $A_1,A_2$ be abelian surfaces. Then, the product of Kummer surfaces
\[
g: \Kum (A_1)\times \Kum(A_2)\to (A_1/\{\pm 1\})\times (A_2/\{\pm 1\})
\]
is the only crepant resolution of $(A_1/\{\pm 1\})\times (A_2/\{\pm 1\})$.
\end{claim}

\begin{proof}[Proof of the claim]
Let $g_i:\Kum (A_i)\to A_i/\{\pm 1\}$ be the minimal resolution 
and $\bar E_{i,1},\cdots ,\bar E_{i,16}$ the exceptional curves. 
Then 
\[
E_{1,j}=\bar E_{1,j}\times \Kum (A_2)\; ,\; E_{2,j}=\Kum (A_1)\times \bar E_{2,j}
\]
are the exceptional divisors of $g$. If $Z'\to (A_1/\{\pm 1\})\times (A_2/\{\pm 1\})$ 
is another crepant resolution, $Z'$ and $\Kum (A_1)\times \Kum(A_2)$ are
isomorphic in codimension one. Let 
$\xymatrix{\phi: \Kum (A_1)\times \Kum(A_2)\ar@{..>}[r]&Z'}$ be the birational map. 
Let $H'$ be an ample divisor on $Z'$ and $H=\phi _*^{-1}H'$ the strict transform 
on $\Kum (A_1)\times \Kum(A_2)$. 
Every $(K_{\Kum (A_1)\times \Kum(A_2)}+\varepsilon H)$-extremal contraction $h$ must be 
a small contraction and contracts a rational curve contained 
in some $E_{1,j_1}\cap E_{2,j_2}\cong \mathbb P^1\times \mathbb P^1$. 
But, then, $h$ must contract at least one of $E_{1,j_1}$ and $E_{2,j_2}$, which is a contraction. 
Therefore, $\phi$ must be an isomorphism. 
\noqed
\end{proof}
This finishes the proof of Main Theorem. 
\end{proof}


\end{document}